\documentclass{article}

\usepackage{amsmath, amsfonts}
\usepackage{amsthm}

\title{Axioms for a local Reidemeister trace in fixed point and
  coincidence theory on differentiable manifolds}

\author{P. Christopher Staecker}

\newcommand{\adm}{\mathcal C}
\newcommand{\liftinc}{\lift{\imath}}
\newcommand{\coinax}{coincidence of lifts}
\newcommand{\Coinax}{Coincidence of lifts}
\newcommand{\lift}{\widetilde}

\newcommand{\RT}{\text{\textit{RT}}}

\newcommand{\Z}{\mathbb{Z}}
\newcommand{\Reid}{\mathcal{R}}
\DeclareMathOperator{\tr}{tr}
\DeclareMathOperator{\ind}{ind}
\DeclareMathOperator{\id}{id}
\DeclareMathOperator{\Coin}{Coin}
\DeclareMathOperator{\Fix}{Fix}

\newtheorem{thm}{Theorem}
\newtheorem{lem}{Lemma}

\newtheorem{prop}{Proposition}
\newtheorem{axiom}{Axiom}

\usepackage{hyperref}

\theoremstyle{plain}

\begin{document}
\bibliographystyle{hplain}

\maketitle

\begin{abstract}
We give axioms which characterize the local Reidemeister trace for orientable
differentiable manifolds.
The local Reidemeister trace in fixed point theory is already
known, and we provide both uniqueness and existence results for the
local Reidemeister trace in coincidence theory.
\end{abstract}

\section{Introduction}
The Reidemeister trace is a fundamental invariant in topological fixed
point theory, generalizing both the Lefschetz and Nielsen
numbers. It was originally defined by Reidemeister in \cite{reid36}.
A more modern treatment, under the name ``generalized Lefschetz
number,'' was given by Husseini in \cite{huss82}. 

If $X$ is a finite connected CW-complex with universal covering space
$\lift X$ and fundamental group $\pi$, then the cellular chain complex
$C_q(\lift X)$ is a free 
$\Z\pi$-module. If $f:X \to X$ is a cellular map and $\lift
f:\lift X \to \lift X$ is a lift of $f$, then the induced map $\lift
f_q: C_q(\lift X) \to C_q(\lift X)$ can be viewed as a matrix with
entries in $\Z\pi$ (with respect to some chosen $\Z\pi$
basis for $C_q(\lift X)$). We then define 
\[ \RT(f,\lift f) = \sum_{q=0}^{\infty} (-1)^q \rho(\tr(\lift f_q)), \]
where $\tr$ is the sum of the diagonal entries of the matrix, and
$\rho$ is the projection into the ``Reidemeister classes'' of
$\pi$. The Reidemeister trace, then, is an element of $\Z R$,
where $R$ is the set of Reidemeister classes. Wecken, in
\cite{weck41}, proved what we will refer to as the \emph{Wecken Trace
  Theorem}, that
\[ \RT(f,\lift f) = \sum_{[\alpha] \in R} \ind([\alpha])\, [\alpha],
\]
where $\ind([\alpha])$ is the index of the Nielsen fixed point class
associated to $[\alpha]$ (see e.g. \cite{jian83}). Thus the number of
terms appearing in the Reidemeister trace with nonzero coefficient is
equal to the Nielsen number of $f$, and by the Lefschetz-Hopf Theorem, the
sum of the coefficients is equal to the Lefschetz number of $f$.

Recent work of Furi, Pera, and Spadini in \cite{fps04} has given a new
proof of the uniqueness of the fixed point index on orientable manifolds
with respect to three natural axioms. In \cite{stae07} their
approach was extended to the coincidence index.
The result is the following theorem:
\begin{thm}\label{indexuniqueness}
Let $X$ and $Y$ be oriented differentiable manifolds of the same
dimension. The coincidence index $\ind(f,g,U)$ of two mappings $f,g:X
\to Y$ over some open set $U \subset X$
is the unique integer-valued function satisfying the following axioms:
\begin{itemize}
\item (Additivity) If $U_1$ and $U_2$ are disjoint open subsets of $U$
  whose union contains all coincidence points of $f$ and $g$ on $U$, then
\[ \ind(f,g,U) = \ind(f,g,U_1) + \ind(f,g,U_2). \]
\item (Homotopy) If $f$ and $g$ are ``admissably homotopic'' to $f'$
  and $g'$, then
\[ \ind(f,g,U) = \ind(f',g',U) \]
\item (Normalization) If $L(f,g)$ denotes the coincidence Lefschetz
  number of $f$ and $g$, then
\[ \ind(f,g,X) = L(f,g). \]
\end{itemize}
\end{thm}

In the spirit of the above theorem, we demonstrate the 
existence and uniqueness of a local Reidemeister trace in
coincidence theory subject to five axioms. A local Reidemeister trace
for fixed point theory was given by Fares and Hart in
\cite{fh94}, but no Reidemeister trace (local or otherwise) has
appeared in the literature for coincidence theory. 

We note that recent work by Gon\c calves and Weber in \cite{gw06}
gives axioms for the Reidemeister trace in fixed point theory using
entirely different methods. Their work uses no locality properties,
and is based on axioms for the Lefschetz number by Arkowitz and
Brown in \cite{ab04}.

In Section \ref{prelim} we present our axiom set, and we prove the
uniqueness in coincidence theory in Section \ref{Suniqueness}. In the
special case of local fixed point theory, we can obtain a slightly
stronger uniqueness result which we discuss in Section
\ref{fixpt}. Section \ref{existence} is a demonstration of the existence
in the setting of coincidence theory. 

This paper contains pieces of the author's doctoral dissertation. The
author would like to thank his dissertation advisor Robert F. Brown
for assistance with both the dissertation work and with this paper. The
author would also like to thank Peter Wong, who guided the early
dissertation work and interested him in the coincidence Reidemeister trace.

\section{The Axioms} \label{prelim}
Throughout the paper, unless otherwise stated, let $X$ and $Y$ denote
connected orientable differentiable manifolds of the same
dimension. All maps $f,g:X \to Y$ will be assumed to be
continuous. The universal covering spaces of $X$ and $Y$ 
will be denoted $\lift X$ and $\lift Y$ with projection maps $p_X:
\lift X \to X$ and $p_Y:\lift Y \to Y$. A \emph{lift} of some map
$f:X \to Y$ is a map $\lift f:\lift X \to \lift Y$ with $p_Y \circ
\lift f = f \circ p_X$. 

Let $f,g:X \to Y$ be maps, with induced homomorphisms $\phi,\psi:\pi_1(X)
\to \pi_1(Y)$ respectively. We will view elements of $\pi_1(X)$ and
$\pi_1(Y)$ as covering transformations, so that for any $\lift x \in
\lift X$ and $\sigma \in \pi_1(X)$, we have $\lift f(\sigma \lift x) =
\phi(\sigma) \lift f(\lift x)$ and $\lift g(\sigma\lift x) =
\psi(\sigma)\lift g(\lift x)$. 

We will partition the elements of $\pi_1(Y)$ into equivalence classes
defined by the ``doubly twisted conjugacy'' relation:
\[ \alpha \sim \beta \iff \alpha = \psi(\sigma)^{-1}\beta
\phi(\sigma). \]
The equivalence classes with respect to this relation
(denoted e.g. $[\alpha]$) are called \emph{Reidemeister classes}. The
set of Reidemeister classes is denoted $\Reid[f,g]$. 

For any set $S$, let $\Z S$ denote the free abelian group generated by
$S$, whose elements we write as sums of elements of $S$ with integer
coefficients. For any such abelian group, there is a homomorphism
$c:\Z S \to \Z$ defined as the sum of the coefficients:
\[ c\left( \sum_i k_i s_i \right) = \sum_i k_i, \]
for $s_i \in S$ and $k_i \in \Z$, and $i$ ranging over a finite
set. 

For some maps $f,g:X \to Y$ and an open subset $U\subset X$, let 
\[ \Coin(f,g,U) = \{ x\in U \mid f(x) = g(x) \}. \]
We say that the triple $(f,g,U)$ is \emph{admissable} if $\Coin(f,g,U)$ is
compact. Two triples $(f,g,U)$ and $(f',g',U)$ are \emph{admissably homotopic}
if there is some pair of homotopies $F_t,G_t:X \times [0,1] \to X$ of
$f,g$ to $f',g'$ with $\{ (x,t) \in U \times [0,1] \mid F_t(x) =
G_t(x) \}$ compact. 

Let $\adm(X,Y)$ be the set of \emph{admissable tuples}, all
tuples of the form $(f, \lift f, g, \lift g, U)$ where $f,g:X \to Y$
are maps, $(f,g,U)$ is an admissable triple, and $\lift f$ and $\lift
g$ are lifts of $f$ and $g$. 

Let $(f, \lift f, g, \lift g, U), (f', \lift f', g', \lift g', U) \in
\adm(X,Y)$ with $(f,g,U)$ admissably homotopic to $(f',g',U)$ by
homotopies $F_t, G_t$. By the homotopy lifting property, there are
unique lifted homotopies $\lift F_t, \lift G_t: \lift X \times [0,1]
\to \lift Y$ with $\lift F_0 = \lift f$ and $\lift G_0 = \lift g$. If
we additionaly have $\lift F_1 = \lift f'$ and $\lift G_1 = \lift g'$,
then we say that the tuples $(f, \lift f, g, \lift g, U)$ and
$(f',\lift f', g', \lift g', U)$ are \emph{admisssably homotopic}.

Throughout the following, let $\RT$
be any function which to an admissable tuple $(f,\lift f, g, \lift g, U) \in
\adm(X,Y)$ associates an element of $\Z\Reid[f,g]$. Our first three
axioms for the local Reidemeister trace are modeled after the axioms
of Theorem \ref{indexuniqueness}.

\begin{axiom}[Additivity] Given $(f, \lift f, g,
  \lift g, U) \in \adm(X,Y)$, if $U_1$ and $U_2$ are disjoint open subsets
  of $U$ with $\Coin(f,g,U) \subset U_1 \cup U_2$, then
\[ \RT(f,\lift f, g, \lift g, U) = \RT(f, \lift f,  g, \lift g, U_1) +
  \RT(f, \lift f, g, \lift g, U_2). \]
\end{axiom}

\begin{axiom}[Homotopy] If $(f, \lift f, g, \lift g, U)$ and
  $(f',\lift f', g',\lift g', U)$ are admissably homotopic admissable
  tuples,  then
\[ \RT(f, \lift f, g, \lift g, U) = \RT(f', \lift f', g', \lift g',
U). \]
\end{axiom}

\begin{axiom}[Normalization]
If $(f,\lift f,g, \lift g, X) \in \adm(X,Y)$, then 
\[ c(\RT(f,\lift f, g, \lift g, X)) = L(f,g), \]
where $L(f,g)$ is the Lefschetz number of $f$ and $g$.
\end{axiom}

We will require one additional axiom to make some connections with
Nielsen theory, based on a well-known property of the Reidemeister
trace:
\begin{axiom}[Lift invariance]
For any $(f, \lift f, g, \lift
g, U) \in \adm(X,Y)$, and any $\alpha, \beta \in \pi_1(Y)$ we have
\[ c(\RT(f, \lift f, g, \lift g, U)) = c(\RT(f, \alpha \lift f, g,
\beta \lift g, U)). \]
\end{axiom}

The four axioms above are enough to demonstrate some relationships
between $\RT$ and the coincidence index.

\begin{prop} \label{coeffs} If $\RT$ satisfies the homotopy,
additivity, normalization, and lift invariance axioms, then 
\[ c(\RT(f, \lift f, g, \lift g, U)) = \ind(f,g,U) \]
for any $(f, \lift f, g, \lift g, U) \in \adm(X,Y)$, where $\ind$
denotes the coincidence index (see \cite{gonc05}).
\end{prop}
\begin{proof}
Let $\omega = c\circ \RT: \adm(X,Y) \to \Z$.
By the lift invariance axiom, $\omega$ is independent of the
choice of lifts. Thus $\omega$ can be viewed as a function
from the set of all admissable \emph{triples} to $\Z$. It is clear
that $\omega$ satisfies the three axioms of Theorem
\ref{indexuniqueness}, since they 
are implied by our additivity, homotopy, and normalization
axioms for $\RT$ (disregarding the lift parameters). Thus $\omega$
is the coincidence index.
\end{proof}

\begin{prop}\label{coinprop}
If $\RT$ satisfies the additivity, homotopy, normalization, and lift
invariance axioms
and $c(\RT(f,\lift f, g, \lift g, U)) \neq 0$, then there is some
$\sigma \in \pi_1(Y)$ such that $\sigma \lift f$ and $\lift g$ have a
coincidence on $p_X^{-1}(U)$.
\end{prop}
\begin{proof} 
By Proposition \ref{coeffs}, if $c(\RT(f,\lift f, g, \lift g, U)) \neq 0$
then $\ind(f,g,U) \neq 0$, and so $f$ and $g$ have a coincidence on
$U$. Let $x \in U$ be this coincidence point, and choose $\lift x \in
p_X^{-1}(x)$. Then since $\lift f$
and $\lift g$ are lifts, the points $\lift f(\lift x)$ and $\lift
g(\lift x)$ will project to the same point of $Y$ by $p_Y$. Thus there
is some covering transformation $\sigma$ with $\sigma \lift f(\lift x)
= \lift g(\lift x)$.
\end{proof}

The four axioms given above are not sufficient to uniquely
characterize the Reidemeister trace in fixed point or coincidence
theory. For instance, the function defined by
\[ T(f,\lift f, g, \lift g, U) = \ind(f,g,U)[1], \]
where [1] is the Reidemeister class of the trivial element $1 \in \pi_1(Y)$,
satisfies all of the axioms above, but provides none of the
expected data concerning $\Reid[f,g]$, and so that function cannot be the
Reidemeister trace.

An additional axiom is needed, one which somehow indicates the
elements of $\Reid[f,g]$ which are to appear in the Reidemeister
trace. Our final axiom is a sort of strengthening of Proposition
\ref{coinprop}, which specifies the Reidemeister data associated to
the coincidence points. 

\begin{axiom}[\Coinax] If $[\alpha]$ appears with nonzero
  coefficient in $\RT(f,\lift f,
g, \lift g, U)$, then $\alpha \lift f$ and $\lift g$ have a
coincidence on $p_X^{-1}(U)$. 
\end{axiom}

Any function $\RT$ which to a tuple $(f,\lift f, g, \lift g,
U) \in \adm(X,Y)$ associates an element of $\Z\Reid[f,g]$, and 
satisfies the additivity, homotopy, normalization, lift
invariance, and \coinax\ axioms we will call a \emph{local
Reidemeister trace}. Our main result (Theorem \ref{uniqueness})
states that there is a unique such function.

\section{Uniqueness} \label{Suniqueness}
Let $(f, \lift f, g, \lift g, U) \in \adm(X,Y)$,
let $\lift U = p_X^{-1}(U)$, and
let 
\[ C(\lift f, \lift g, \lift U, [\alpha]) = p_X(\Coin(\alpha \lift f, \lift
g, \lift U)).\] 
For each $\alpha$ we have $C(\lift f, \lift g, \lift
U, [\alpha]) \subset \Coin(f,g,U)$, and such coincidence sets are
called \emph{coincidence 
classes}. That these classes are well defined is a consequence of the
following lemma, which appears in slightly different language as Lemma
2.3 of \cite{dj93}. 

\begin{lem}\label{coinlift}
Let $\alpha, \beta \in \pi_1(Y)$, maps $f,g:X \to Y$, and an open subset
$U \subset X$ be given. Then:
\begin{itemize}
\item $[\alpha] = [\beta]$ if and only if 
\[ p_X \Coin(\alpha\lift f, \lift g, \lift U) = p_X \Coin(\beta\lift
f, \lift g, \lift U) \]
for any lifts $\lift f, \lift g$. 
\item If $[\alpha] \neq 
[\beta]$, then $p_X \Coin(\alpha \lift f, \lift g, \lift U)$ and $p_X
\Coin(\alpha \lift f, \lift g, \lift U)$ are disjoint for any lifts
$\lift f,\lift g$. 
\end{itemize}
\end{lem}

Given the above notation, the \coinax\ axiom could be restated
as follows: If $[\alpha]$ appears with nonzero coefficient in
$\RT(f,\lift f, g, \lift g, U)$, then $C(\lift f, \lift g, \lift U,
[\alpha])$ is nonempty. For each coincidence point $x$ in $U$, define
$[x_{\lift f,\lift g}] \in \Reid[f,g]$ as that class $[\alpha]$ for
which $x \in C(\lift f, \lift g, \lift U, [\alpha])$.

\begin{thm} \label{weckentrace}
If $\RT$ is a local Reidemeister trace and $\Coin(f,g,U)$ is a set of
isolated points, then
\[ \RT(f,\lift f, g, \lift g, U) = \sum_{x \in \Coin(f,g,U)} \ind(f,g,U_x)
   [x_{\lift f,\lift g}], \]
where $U_x$ is an isolating neighborhood for the coincidence point $x$.
\end{thm}
\begin{proof}
By the additivity property, we need only show that
\[ \RT(f,\lift f, g, \lift g, U_x) = \ind(f, g, U_x) [x_{\lift f,\lift
g}]. \]
First, we observe that no element
of $\Reid[f,g]$ other than $[x_{\lift f, \lift g}]$ appears as a term
with nonzero coefficient in 
$\RT(f,\lift f, g, \lift g, U_x)$: 
If some $[\beta]$ does appear with nonzero coefficient, then we know
by the coincidence of lifts axiom that $\beta \lift f$ and $\lift g$
have a coincidence on $\lift U_x = p_X^{-1}(U_x)$. Projection of this coincidence
point gives a coincidence point in $U_x$ which necessarily must be
$x$, since $x$ is the only coincidence point in $U_x$. Thus $x \in p_X
\Coin(\beta \lift f, \lift g, \lift U_x)$, which means that $[\beta] =
[x_{\lift f, \lift g}]$.

Since $[x_{\lift f, \lift g}]$ is the only element of
$\Reid[f,g]$ appearing in $\RT(f,\lift f, g, \lift g, U)$, we have
\[ \RT(f,\lift f, g, \lift g, U_x) = k [x_{\lift f, \lift g}] \]
for some $k\in \Z$ (possibly $k=0$). Proposition \ref{coeffs} says
that the coefficient sum must equal the index, and so $k =
\ind(f, g, U_x)$ as desired.
\end{proof}

The above is a strong result for maps whose coincidence sets are
isolated. In order to leverage this result for arbitrary maps, we will
make use of a technical lemma, a combination of Lemmas 13 and 
15 from \cite{stae07}.

\begin{lem} \label{mfldisolation} 
Let $(f,g,U)$ be an admissable triple, and let $V \subset U$ be an
open subset containing $\Coin(f,g,U)$ with compact closure $\bar V
\subset U$. Then $(f,g,V)$ is admissably homotopic to an admissable
triple $(f',g',V)$, where $f'$ and $g'$ have isolated coincidence
points in $V$. 
\end{lem}

The above lemma is used to approximate any maps by maps having
isolated coincidence points, and we obtain our uniqueness theorem:
\begin{thm}\label{uniqueness}
There is at most one local Reidemeister trace defined on $\adm(X,Y)$.
\end{thm}
\begin{proof}
Let $\RT$ be local Reidemeister trace, and take $(f, \lift f, g, \lift
g, U) \in 
\adm(X,Y)$. Then by Lemma \ref{mfldisolation} there is an open subset $V
\subset U$ with $\Coin(f,U) \subset V$ and maps $f', g'$ with isolated
coincidence points with $(f,g,V)$ admissably homotopic to
$(f',g',V)$. Then by the homotopy axiom there are lifts $\lift
f', \lift g'$ of $f$ and $g$ with 
\[ \RT(f,\lift f, g, \lift g, U) = \RT(f', \lift f', g', \lift g', V). \]

The coincidence points of $f'$ and $g'$ in $V$ are isolated, so we have
\[ \RT(f, \lift f, g, \lift g, U) = \sum_{x\in \Coin(f',g',V)}
\ind(f',g', V_x) [x_{\lift f', \lift g'}], \]
where $V_x$ is an isolating neighborhood of the coincidence point $x$.
This gives an explicit formula for the computation of $\RT(f, \lift f,
g, \lift g, U)$. The only choice made in the computation is of the
admissable homotopy to $(f', g', V)$, but any alternative choice
must give the same local Reidemeister trace by the homotopy axiom. Thus all
local Reidemeister traces must be computed in the same way, giving the same
result, which means that there can be only one. 
\end{proof}

\section{Uniqueness in fixed point theory} \label{fixpt}
In the special case where $Y=X$ and $g$ is taken to be the
identity map $\id:X \to X$, the above method can be used
with slight modifications to prove a uniqueness result for the local
Reidemeister trace in the fixed point theory of possibly nonorientable
manifolds. 

We have not in this paper made explicit use of the orientability
hypothesis, but it is a necessary hypothesis for the uniqueness of the
coincidence index in Theorem \ref{indexuniqueness}, which was used in Proposition
\ref{coeffs}. An accounting of orientations is needed in coincidence
theory to distinguish between points of index $+1$ and index $-1$
(though see \cite{dj93} for an approach to an index for nonorientable
manifolds, which does not always give an integer). Orientability is
not needed in local fixed point theory, since 
the notion of an orientation preserving selfmap is well-defined
locally, even on a manifold with no global orientation. Thus the
uniqueness of the fixed point index in \cite{fps04} does not require
orientability, and we will not require it here.

Let $\adm(X)$ be the set of all tuples of the form $(f,\lift f,
\liftinc, U)$,
where $f:X \to X$ is a selfmap, $\lift f: \lift X \to \lift X$ is a
lift of $f$, the map $\liftinc: \lift X \to \lift X$ is a lift of
the identity map, and $U$ is an open subset of $X$ with compact fixed point set
$\Fix(f,U) = \Coin(f,\id,U)$. Let $\Reid[f] = \Reid[f,\id]$. Two
tuples $(f,\lift f, \liftinc, 
U)$ and $(f',\lift f',\liftinc, U)$ are said to be admissably
homotopic if there is some homotopy $F_t$ of $f$ to $f'$ with $\{
(x,t) \mid F_t(x) = x \}$ compact, and $F_t$ lifts to a homotopy of
$\lift f$ to $\lift f'$. 

Our uniqueness theorem is then:
\begin{thm}\label{fixptuniqueness}
If $X$ is a (possibly nonorientable) differentiable manifold, then there is
a unique function taking an admissable tuple $(f,\lift f,\liftinc, U)$ to an
element of $\Z\Reid[f]$ satisfying the following axioms:
\begin{itemize}
\item (Additivity) If $U_1$ and $U_2$ are disjoint open subsets of $U$
  with $\Fix(f,U) \subset U_1 \cup U_2$, then
\[ \RT(f,\lift f,\liftinc, U) = \RT(f, \lift f,\liftinc, U_1) +
  \RT(f, \lift f,\liftinc, U_2) \]
\item (Homotopy) If $(f,\lift f,\liftinc, U)$ is admissably homotopic to
  $(f',\lift f',\liftinc, U)$, then
\[ \RT(f,\lift f,\liftinc, U) = \RT(f',\lift f',\liftinc, U) \]
\item (Weak normalization) If $f$ is a constant map, then
\[ c (\RT(f,\lift f,\liftinc, U)) = 1\]
\item (Lift invariance) For any $\alpha, \beta \in \pi_1(X)$, we have
\[ c(\RT(f,\lift f, \liftinc, U)) = c(\RT(f, \alpha\lift f, \beta
\liftinc, U)) \]
\item (\Coinax) If $[\alpha]$ appears with nonzero coefficient
  in $\RT(f,\lift f,\liftinc, U)$,
  then $\alpha \lift f$ and $\liftinc$ have a coincidence point on
  $p_X^{-1}(U)$.  
\end{itemize}
\end{thm}
\begin{proof}
First we note that a result analagous to Proposition \ref{coeffs} can
be obtained in the fixed point setting using only the weak
normalization axiom: Using
the three axioms of \cite{fps04}, which make use of an appropriately
weakened normalization axiom, we see that $c\circ \RT$ is the fixed
point index.

Then letting $g = \id$ in the proof of Theorem \ref{weckentrace}, we
have that, if $f$ has isolated fixed points,
\[ \RT(f,\lift f, \liftinc, U) = \sum_{x \in \Fix(f,U_x)} \ind(f,U_x) [x_{\lift f,\liftinc}],
\]
where $\ind$ denotes the fixed point index, and $U_x$ is an isolating
neighborhood for the fixed point $x$.

A fixed point version of Lemma \ref{mfldisolation} can be found in
Lemmas 4.1 and 3.3 of \cite{fps04}, and the proof of Theorem
\ref{uniqueness} can be mimicked to obtain our uniqueness result.
\end{proof}

Note that the uniqueness in fixed point theory requires only a
weakened version of the normalization axiom. A uniqueness result for
coincidence theory using only the weak 
normalization axiom can be obtained if we restrict ourselves to
self-maps of a particular (not necessarily orientable) manifold. This
would use a proof similar to the above, using results from Section 5 of
\cite{stae07}. 

\section{Existence} \label{existence}
The existence of a local Reidemeister trace in fixed point theory for
connected finite 
dimensional locally compact polyhedra is established by Fares and Hart
in \cite{fh94}. There, the slightly more general local $H$-Reidemeister
trace is defined, called ``the local generalized $H$-Lefschetz
number''. An extension of this paper to the mod $H$ theory
would not be difficult. The fact that the mod $H$ Reidemeister classes
are unions of ordinary Reidemeister classes allows the same results to
be obtained without substantial modifications. 

In \cite{fh94}, the additivity and homotopy
axioms are proved in Proposition 3.2.9 and Proposition 3.2.8,
respectively. A strong version of the lift invariance axiom (see our
Theorem \ref{liftcoeffs}) is proved in Proposition 3.2.4. The
\coinax\ axiom is not stated explicitly by Fares and Hart, but is a
straightforward consequence of their trace-like definition
(if some $[\alpha]$ has nonzero coefficient in the Reidemeister trace,
it neccesarily comes from some simplex in the covering space containing
a fixed point of $\alpha\lift f$). A result analogous to the Wecken
Trace Theorem (which trivially implies the normalization and weak
normalization axioms) is given in Theorem 
3.3.1.

No Reidemeister trace for coincidence theory, either local or global,
has appeared previously in the literature. The proof of Theorem
\ref{uniqueness} furnishes the appropriate definition, as follows:
Given an admissable tuple $(f,\lift f, g, \lift g,
U)$, we find (by Lemma \ref{mfldisolation}) an admissably homotopic
tuple $(f',\lift f', g', \lift g', V)$ with isolated coincidence
points, and we define
\[ \RT(f,\lift f,g, \lift g, U) = \sum_{x \in \Coin(f',g',V)}
\ind(f',g',V_x) [x_{\lift f',\lift g'}], \]
where $V_x$ is an isolating neighborhood for the coincidence point $x$.
The above is well defined provided that it is independent of the
choice of the admissably homotopic tuple. This is ensured by the
following lemma:
\begin{lem}
If $(f,\lift f, g, \lift g, U)$ and $(f', \lift f', g', \lift g', U)$
are admissably homotopic tuples with isolated coincidence points, then 
\[ \sum_{x \in \Coin(f,g,U)} \ind(f,g,U_x)[x_{\lift f, \lift g}] =
  \sum_{x' \in \Coin(f',g',U)} \ind(f',g',U_{x'}) [x'_{\lift f',\lift
  g'}], \]
where $U_x$ is an isolating neighborhood for the coincidence point $x
\in \Coin(f,g,U)$, and $U_{x'}$ is an isolating neighborhood of the
coincidence point $x' \in \Coin(f',g',U)$.
\end{lem}
\begin{proof}
We define the index of a coincidence class $C$ of $f$ and $g$ as follows: 
\[ \ind C = \sum_{x \in C} \ind(f,g,U_x). \]
A class is called \emph{essential} if its index is nonzero.

Since $f$ and $g$ are homotopic to $f'$ and $g'$, we have $\Reid[f,g]
= \Reid[f',g']$. Call this common set of Reidemeister classes $R$.
Letting $\lift U = p_X^{-1}(U)$, the statement of the Lemma is equivalent to
\[ \sum_{[\alpha] \in R} \ind C(\lift f,\lift g, \lift U, [\alpha]) [\alpha] =
\sum_{[\alpha] \in R} \ind C(\lift f',\lift g', \lift U, [\alpha]) [\alpha], \]
and we need only show that $\ind
C(\lift f,\lift g,\lift U,[\alpha]) = \ind C(\lift f',\lift g', \lift
U, [\alpha])$ for any $[\alpha]$. We will prove this using Brooks's notion of
homotopy-relatedness of coincidence classes, exposited in detail in
\cite{broo67} and briefly in \cite{bb69}. 

Let $F_t, \lift F_t, G_t, \lift G_t$ be homotopies realizing the
admissable homotopy of 
$(f,\lift f, g, \lift g, U)$ and $(f',\lift f', g', \lift g', U)$. Two
coincidence points $x \in \Coin(f,g,U)$ and $x' \in \Coin(f',g',U)$
are \emph{$(F_t,G_t)$--related} if there is some path $\gamma(t)$ in
$X$ connecting $x$ to $x'$ such that the paths $F_t(\gamma(t))$ and
$G_t(\gamma(t))$ are homotopic in $Y$ as paths with fixed
endpoints. Two coincidence classes are related if at least one point
of one is related to at least one point of the other.
Theorem II.22 of \cite{broo67} shows that the notion of
$(F_t,G_t)$-relatedness gives a bijective correspondence between
the essential coincidence classes of $(f,g)$ and those of
$(f',g')$. Theorem IV.24 of \cite{broo67} further shows that any two
such related classes will have the same index.

What remains is an elementary argument using covering-space
theory. Let $C = C(\lift f, \lift g, \lift U, [\alpha])$, and let $C'$ be the
unique coincidence class of $(f',g')$ which is $(F_t,G_t)$-related to
$C$. We need only show that $C' = C(\lift f', \lift g', \lift U, [\alpha])$,
and thus (since homotopy-relatedness preserves the index) that $\ind
C(\lift f, \lift g, \lift U, [\alpha]) = \ind C(\lift f', \lift g', \lift U, [\alpha])$.

Choose a point $x \in C$, and let $x'$ be a point
in $C'$ which is $(F_t,G_t)$ related to $x$. Then there is some path
$\gamma$ in $X$ from $x$ to $x'$ with $F_t(\gamma(t))$ homotopic to
$G_t(\gamma(t))$. 
Let $\lift x$ be some point with $p_X(\lift x) = x$ and $\alpha
\lift f (\lift x) = \lift g(\lift x)$. We can lift $\gamma$ to a path
$\lift \gamma$ in $\lift X$ starting at $\lift x$. Since
$F_t(\gamma(t))$ is homotopic to $G_t(\gamma(t))$, we will have $\lift
F_t(\lift \gamma(t))$ homotopic to $\lift G_t(\lift \gamma(t))$, which
in particular means that they will have the same endpoint. This
common endpoint is $\alpha \lift f'(\lift \gamma(1)) = \lift
g'(\lift \gamma(1))$, which must project by $p_X$ to the point
$x'$. Thus $x' \in p_X(\Coin(\alpha \lift f', \lift g', \lift U))$,
and so $C' = C(\lift f', \lift g', \lift U, [\alpha])$, as desired.
\end{proof}

We have thus produced a meaningful definition of a local coincidence
Reidemeister trace on orientable differentiable manifolds of the same
dimension, and the proof above suffices to give:
\begin{thm}[Wecken Coincidence Trace Theorem]\label{wtt}
Let $\RT$ be the unique local coincidence Reidemeister trace
satisfying our five axioms. Then for any $(f, \lift f, g, \lift g, U)
\in \adm(X,Y)$ with $\lift U = p_X^{-1}(U)$, we have 
\[ \RT(f,\lift f, g, \lift g, U) = \sum_{[\alpha] \in \Reid[f,g]} \ind
C(\lift f, \lift g, \lift U, [\alpha]) [\alpha]. \]
\end{thm}

In conclusion we prove a stronger form of the lift invariance axiom, a
coincidence version of a well-known property of the Reidemeister trace.
\begin{thm} \label{liftcoeffs} Let $\RT$ be the unique local
coincidence Reidemeister trace satisfying our five axioms. If 
\[ \RT(f,\lift f, g, \lift g, U) = \sum_{[\sigma] \in \Reid[f,g]}
k_{[\sigma]} [\sigma] \]
for $k_{[\sigma]} \in \Z$, then for any $\alpha, \beta \in \pi_1(Y)$,
we have
\[
\RT(f, \alpha \lift f, g, \beta \lift g, U) = \sum_{[\sigma] \in
  \Reid[f,g]} k_{[\sigma]} [\beta \sigma \alpha^{-1}]. \]
\end{thm}
\begin{proof}
Letting $\lift U = p_X^{-1}(U)$, by Theorem \ref{wtt} we know
that $k_{[\sigma]} = \ind C(\lift f, \lift g, \lift U, [\sigma])$. Then we have
\[
C(\alpha \lift f, \beta \lift g, [\sigma]) = p_X \Coin(\sigma
\alpha \lift f, \beta \lift g, \lift U) = p_X \Coin(\beta^{-1} \sigma \alpha
\lift f, \lift g, \lift U) 
= C(\lift f, \lift g, [\beta^{-1} \sigma \alpha]), \]
and thus $\ind C(\alpha \lift f, \beta \lift g, \lift U, [\sigma]) =
k_{[\beta^{-1} \sigma \alpha]}$. Now by Theorem \ref{wtt} again, we have 
\begin{align*} 
\RT(f,\alpha \lift f, g, \beta \lift g, \lift U) &= \sum_{[\sigma] \in
\Reid[f,g]} \ind C(\alpha \lift f, \beta\lift g, [\sigma]) [\sigma] =
\sum_{[\sigma] \in \Reid[f,g]} k_{[\beta^{-1} \sigma \alpha]} [\sigma] \\
&= \sum_{[\gamma] \in \Reid[f,g]} k_{[\gamma]} [\beta \gamma
\alpha^{-1}], 
\end{align*}
as desired.
\end{proof}

\end{document}